\documentclass[12pt,a4paper]{article}

 \usepackage{ifpdf}

 \usepackage{latexsym}                   
 \usepackage{epsfig}                     
 \usepackage{graphicx}
 \usepackage{amssymb, amsmath, amsthm}   
 \usepackage{enumerate}                  
 \usepackage{mathrsfs}                   
 \usepackage[plain]{algorithm}
 \usepackage{geometry}                                                                          
 \usepackage{verbatim}                                                                           

 \usepackage[round]{natbib}
 \usepackage{setspace}

 \newtheorem{thm}{Theorem}[section]

 \newtheorem{cor}[thm]{Corollary}

 \theoremstyle{definition}
 
 \theoremstyle{definition}
 \newtheorem{rem}[thm]{Remark}
 \theoremstyle{definition}

 \newcommand{\bE}{{\mathbb E}}
 
 \newcommand{\bN}{{\mathbb N}}
 \newcommand{\bP}{{\mathbb P}}

 \newcommand{\cT}{{\mathcal T}}

 \newcommand {\e}      {\mbox{\bf e}}

\begin{document}
%
%

 \thispagestyle{plain}
   \title{Lineage-through-time plots of birth-death processes}
  
   \author{Tanja Gernhard, Dennis Wong \\ 
     \date{}
    {\small Department of Mathematics, Kombinatorische Geometrie (M9), TU M\"{u}nchen } \\ 
    {\small Bolzmannstr. 3, 85747 Garching, Germany} \\ 
    {\small Phone +49 89 289 16882, gernhard@ma.tum.de}\\
   }
   \maketitle


 \doublespacing

 \begin{abstract}
We calculate the density and expectation for the number of lineages in a reconstructed tree with $n$ extant species.
This is done with conditioning on the age of the tree as well as with assuming a uniform prior for the age of the tree.
 \end{abstract}


\section{Introduction}
There are a variety of methods to extract information relevant to macro-evolutionary process from phylogenies (e.g. imbalance [e.g. Heard 1992]; Gamma [e.g. Pybus and Harvey 2001]). One popular approach are lineages through time plots (LTT plots). An LTT plot is a plot of lineage accumulation through time translated from a dated phylogeny (Nee et al. 1992). An LTT plot can be used to test the plausibility of a model of macroevolution for any particular clade, the LTT plot for the clade of interest can be compared to an expectation generated from a model (for a review of such models see Mooers et al. 2007 [book chapter], or Hartmann et al. in press). The expected shape can be obtained either through simulation or analytical approaches. Simulations are a simple and therefore attractive approach to developing models of macroevolution, but their use can be trecherous (see Hartmann et al) for an illuminating discussion.
Analytical approaches offer a computational advantage over simulation, but even simple models are quite challenging to analyze.
Here, we add to the knowledge base of analytical approaches with respect to LTT plots.

We consider constant rate birth and death processes \citep{Feller1968}. The birth rate is $\lambda$, the death rate is $\mu$. We define $\rho := \mu / \lambda$ and $\delta := \lambda - \mu$. Constant rate birth and death processes are a popular null model ...
bla bla bla

Birth and death process is conditioned such that we obtain $n$ species today. A tree with both extinct and extant species is a complete tree, while a reconstructed tree is the complete tree where all lineages are removed.  

We will need the following functions as defined in \citet{Nee1994},
\begin{eqnarray}
P(t) &=& \frac{\lambda-\mu}{\lambda-\mu e^{-(\lambda-\mu)t}}, \label{EqnPt} \\
u(t) &=& \lambda \frac{1-e^{-(\lambda-\mu)t}}{{\lambda-\mu e^{-(\lambda-\mu)t}}}. \label{Eqnut}
\end{eqnarray}

\section{LTT plots for trees of known age} \label{SecLTT}

In a lineage-through-time (LTT) plot, we plot the time vs. the number of species at that time.
For a reconstructed tree, in \citet{Nee1994} the expected LTT plot is given analytical after a time $t$.

However, when analyzing the data, we have trees on a given number of species, $n$. The aim of this section is to calculate the density and expectation for the number of species at time $\sigma t$, $\sigma \in \{0,1\}$ in a reconstructed tree at time $t$ after origin. We call this random variable $M_{\sigma,t}$. We condition $M_{\sigma,t}$ on having $M_{1,t}=n$, i.e. having $n$ species today.

\begin{thm} \label{ThmPmgivenn}
Let today be time $t$ and assume we have $n$ species today.
The probability that at time $\sigma t$ we have $m$ species in the reconstructed tree is
\begin{equation}
\bP[M_{\sigma,t}=m|M_{1,t}=n] = \left\{
\begin{array}{ll}
     {n-1 \choose m-1} \frac{ f(\sigma,t, \rho, \delta)^{m-1}}{\left(1+f(\sigma,t,\rho, \delta) \right)^{n-1}}  & \hbox{if $m \leq n$} \\
    0 & \hbox{else} \\
\end{array}
\right.
\end{equation}
with
 $f(\sigma,t,\rho,\delta) =  (1- \rho)  \frac{(1-e^{- \sigma \delta t})e^{- (1-\sigma)  \delta t}}{(1-e^{-(1-\sigma)\delta t})(1 - \rho e^{-\delta t}) }$.
\end{thm}
\begin{proof}
Since we are considering reconstructed trees, we obviouly have $\bP[M_{\sigma,t}=m|M_{1,t}=n]=0$ if $m>n$.
For $m \leq n$, we have with Bayes' law,
\begin{equation}
\bP[M_{\sigma,t}=m|M_{1,t}=n] = \bP[M_{1,t}=n|M_{\sigma,t}=m] \frac{\bP[M_{\sigma,t}=m]}{\bP[M_{1,t}=n]}. \label{EqnBayes}
\end{equation}
The probability that a lineage in the reconstructed tree at time $\sigma t$ has $m$ descendants today, at time $t$, is
$$P_1(\sigma,m) = (1-u((1-\sigma)t)) u((1-\sigma)t)^{m-1}$$
which is established in \citet{Nee1994}, Equation (4).
Therefore, with $N= \sum_{m=n}^{\infty} \bP[M_{1,t}=n|M_{\sigma,t}=m]$, and $\e=(1,1, \ldots,1)^T$, we get
\begin{eqnarray*}
\bP[M_{1,t}=n|M_{\sigma,t}=m] &=& \frac{1}{N} \sum_{\substack{i \in \bN \\ i^T \e = n}} \prod_{k=1}^m P_1(\sigma,i_k) \\
&=& \frac{1}{N} \sum_{\substack{i \in \bN \\ i^T \e = n}} \prod_{k=1}^m (1-u((1-\sigma)t)) u((1-\sigma)t)^{i_k-1} \\
&=& \frac{1}{N} \sum_{\substack{i \in \bN \\ i^T \e = n}} (1-u((1-\sigma)t))^m u((1-\sigma)t)^{n-m}\\
&=& \frac{1}{N} |\{i \in \bN : i^T \e = n \}| (1-u((1-\sigma)t))^m u((1-\sigma)t)^{n-m}.\\
\end{eqnarray*}
We determine $|\{i \in \bN : i^T \e = n \}|$. For every component of $i$, we have $i_k \geq 1, k = 1, \ldots, m$. So we have to count in how many ways we can distribute
the remaining $n-m$ ones to the $m$ components. Distibuting the $n-m$ ones to $m$ components is equivalent to drawing $n-m$ times from a urn with $m$ different balls
and returning the balls to the urn after a drawing.
From combinatorics, we know that there are ${n-m+m-1 \choose n-m}={n-1 \choose m-1}$ different outcomes. 
So $|\{i \in \bN : i^T \e = n \}| = {n-1 \choose m-1}$. Therefore,
$$\bP[M_{1,t}=n|M_{\sigma,t}=m] = \frac{1}{N} {n-1 \choose m-1} (1-u((1-\sigma)t))^m u((1-\sigma)t)^{n-m}.$$
In \citet{Nee1994}, the authors establish (Equation (9) and (3))
\begin{eqnarray*}
\bP[M_{\sigma,t}=m] &=& \left(1-u(\sigma t) \frac{P(t)}{P(\sigma t)}\right) \left(u(\sigma t) \frac{P(t)}{P(\sigma t)}\right)^{m-1}, \\
\bP[M_{1,t}=n] &=&  P(t)(1-u(t))u(t)^{n-1}.
\end{eqnarray*}
Plugging these equations into Equation (\ref{EqnBayes}) yields
\begin{eqnarray*}
& & \bP[M_{\sigma,t}=m]\\
&=& \frac{1}{N}{n-1 \choose m-1} (1-u((1-\sigma)t))^m u((1-\sigma)t)^{n-m}
\frac{\left(1-u(\sigma t) \frac{P(t)}{P(\sigma t)}\right) \left(u(\sigma t) \frac{P(t)}{P(\sigma t)}\right)^{m-1}}{ P(t)(1-u(t))u(t)^{n-1}} \\
&=& \frac{1}{N}{n-1 \choose m-1} \left(u(\sigma t) \frac{1-u((1-\sigma)t)}{u((1-\sigma)t)} \frac{P(t)}{P(\sigma t)}\right)^{m-1} \\
& &
u((1-\sigma)t)^{n-1} (1-u((1-\sigma)t)) \frac{\left(1-u(\sigma t) \frac{P(t)}{P(\sigma t)}\right)}{ P(t)(1-u(t))u(t)^{n-1}} \\
&=& \frac{1}{N}{n-1 \choose m-1} \left(u(\sigma t) \frac{1-u((1-\sigma)t)}{u((1-\sigma)t)} \frac{P(t)}{P(\sigma t)}\right)^{m-1} g_{\sigma,t,n} 
\end{eqnarray*}
where $g_{\sigma,t,n} = u((1-\sigma)t)^{n-1} (1-u((1-\sigma)t)) \frac{\left(1-u(\sigma t) \frac{P(t)}{P(\sigma t)}\right)}{ P(t)(1-u(t))u(t)^{n-1}}.$
In the following, we determine $N$. Since probabilities add up to $1$, we have $\sum_{m=1}^n \bP[M_{\sigma,t}=m|M_{1,t}=n]=1$.  We have with the binomial theorem,
$$N = N \sum_{m=1}^n \bP[M_{\sigma,t}=m|M_{1,t}=n] = \left(1+ u(\sigma t) \frac{1-u((1-\sigma)t)}{u((1-\sigma)t)} \frac{P(t)}{P(\sigma t)}  \right)^{n-1} g_{\sigma,t,n}.$$
Therefore,
$$ \bP[M_{\sigma,t}=m|M_{1,t}=n] = {n-1 \choose m-1} \frac{\left(u( \sigma t) \frac{1-u((1-\sigma)t)}{u((1-\sigma)t)} \frac{P(t)}{P(\sigma t)}\right)^{m-1}}{\left(1+ u(\sigma t) \frac{1-u((1-\sigma)t)}{u((1-\sigma)t)} \frac{P(t)}{P(\sigma t)}  \right)^{n-1}}$$
We evaluate
\begin{eqnarray*}
 u( \sigma t) \frac{1-u((1-\sigma)t)}{u((1-\sigma)t)} \frac{P(t)}{P(\sigma t)}
&=& (\lambda-\mu) \frac{(1-e^{-(\lambda-\mu) \sigma t})e^{-(\lambda-\mu) ((1-\sigma)t)}}{(1-e^{-(\lambda-\mu)((1-\sigma)t)})(\lambda-\mu e^{-(\lambda-\mu) t}) } 
\end{eqnarray*}
with $P(t)$ and $u(t)$ from Equation (\ref{EqnPt}) and (\ref{Eqnut}). So
$$f(\sigma,t,\rho,\delta) :=  u( \sigma t) \frac{1-u((1-\sigma)t)}{u((1-\sigma)t)} \frac{P(t)}{P(\sigma t)} =   (1-\rho)  \frac{(1-e^{- \sigma \delta t})e^{- (1-\sigma)  \delta t}}{(1-e^{-(1-\sigma)\delta t})(1-\rho e^{-\delta t}) }.$$
Therefore,
$$ \bP[M_{\sigma,t}=m|M_{1,t}=n] = {n-1 \choose m-1}  \frac{ f(\sigma,t,\rho,\delta)^{m-1}}{\left( 1+f(\sigma,t,\rho,\delta) \right)^{n-1}}$$
which establishes the theorem.


\end{proof}
 \begin{rem}
Note that $f(\sigma,t,\rho,\delta) = f(\sigma,\delta t,\rho,1)$. 
Therefore, the conditional distribution $\bP[M_{\sigma,t}=m|M_{1,t}=n]$ with parameters $\rho,\delta$ is the same as $\bP[M_{\sigma,\delta t}=m|M_{1,\delta t}=n]$ with parameters $\rho,1$.
\end{rem}

\begin{cor}
The expectation of $M_{\sigma,t}$ given $M_{1,t}=n$ is
$$ \bE[ M_{\sigma,t} |M_{1,t}=n]=\frac{1+nf(\sigma,t,\rho,\delta)}{1+f(\sigma,t,\rho,\delta)}$$
\end{cor}

\begin{proof}
From Theorem \ref{ThmPmgivenn}, we get
\begin{eqnarray*}
\bE[ M_{\sigma,t} |M_{1,t}=n] &=& \sum_{m=1}^{n} m  \bP[M_{\sigma,t}=m|M_{1,t}=n] \\
&=&   \frac{1}{\left( 1+f(\sigma,t,\rho,\delta) \right)^{n-1}}  \sum_{m=0}^{n-1} (m+1)  {n-1 \choose m}  f(\sigma,t,\rho,\delta)^{m} \\
&=&  \frac{1}{\left( 1+f(\sigma,t,\rho,\delta) \right)^{n-1}} \left[ (f(\sigma,t,\rho,\delta)+1)^{n-1} + \sum_{m=1}^{n-1} m  {n-1 \choose m}  f(\sigma,t,\rho,\delta)^{m}\right] \\
&=&  1 + \frac{ (n-1)f(\sigma,t,\rho,\delta) }{\left( 1+f(\sigma,t,\rho,\delta) \right)^{n-1}} \sum_{m=1}^{n-1}  {n-2 \choose m-1}  f(\sigma,t,\rho,\delta)^{m-1} \\
&=&  1 +  \frac{(n-1) f(\sigma,t,\rho,\delta)  (f(\sigma,t,\rho,\delta)+1)^{n-2}}{\left( 1+f(\sigma,t,\rho,\delta) \right)^{n-1}}  \\
&=&  \frac{1+nf(\sigma,t,\rho,\delta)}{ 1+f(\sigma,t,\rho,\delta) } 
\end{eqnarray*}
which establishes the corollary.
\end{proof}

Note that for a fixed $n$, the conditional expectation $\bE[ M_{\sigma,t} |M_{1,t}=n]$ only depends on $\rho$ and $\delta t$. For $\rho = 0, 1/4,1/2,3/4,1$, $t=10$ and varying values of $\delta$, we calculated the expectation, see Figure \ref{FigExp}. The graph looks quite unfamiliar for an LTT plot of a reconstructed tree since we have concave curves, and the Yule model is for large $\lambda$ more convex than models with extinction.

This has the following reason. Consider the curves for arbitrary $\lambda$ and $\mu=0$. We condition on the age $t$ of the tree. If $\lambda$ is very large, i.e. the process will have more than $n$ lineages at time $t$ with high probability (when not conditioning on $n$), then the most likely trees with $n$ species are the trees where nothing happens at the beginning, and later we have speciation.
If lots of speciation would happen at the beginning, we would later allow all those lineages only speciate very rarely, since we want to end up with $n$ species. This is very unlikely though, since $\lambda$ is big. If at the beginning, the one lineage does not speciate, and after a while, we would have ``normal'' speciation, this is much more likely, since we only force the first lineages to behave unnormal. This yields a very convex LTT plot.

In the case of $\lambda$ being small compared to $t$, we need the early lineages to speciate a lot. Then the later lineages can behave quite normal in order to end up with $n$ lineages today.

\begin{figure}[!h]
\begin{center}
\includegraphics[scale=0.8]{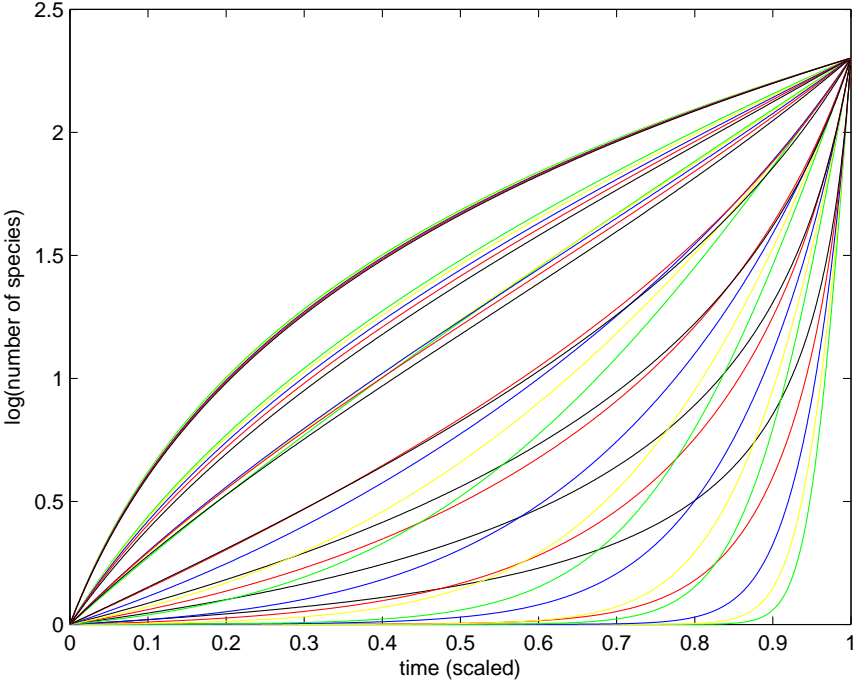}
\caption{Expected number of species given we have $n=10$ species today at time $t=10$. We calculated for $\lambda = 5,2,1,0.5,0.2,0.1,0.01$, from bottom to top. The different colours correspond to
green: $\rho = 0$, yellow: $\rho=1 / 4$, blue: $\rho = 1 /2$, red: $\rho = 3/4 $, black: $\rho = 1$.}
\label{FigExp}
\end{center}
\end{figure}

\subsection{Conditioning on the most recent common ancestor}
So far, we condition on the time of the origin of our tree.
In other situations, we might know the time of the $mrca$ of the extant species opposed to the time of origin.

Let $M^{mrca}_{\sigma,t}$ be the random variable `number of lineages in reconstructed tree at time $\sigma t$ given the time since the $mrca$ is $t$'.
\begin{cor}
For $M^{mrca}_{\sigma,t}$, we have the following conditional density,
\begin{eqnarray*}
\bP[M^{mrca}_{\sigma,t}=m |M^{mrca}_{1,t}=n ] &=& \frac{1}{n-1} \frac{f(\sigma,t,\rho,\delta)^{m-1}}{(1+f(\sigma,t,\rho,\delta)^{n-1}} \sum_{k=1}^{n-1} \sum_{l=1}^{m-1} {k-1 \choose l-1} {n-k \choose m-l}.
\end{eqnarray*}
for $m \leq n$ and $\bP[M^{mrca}_{\sigma,t}=m |M^{mrca}_{1,t}=n ] = 0$ otherwise.
\end{cor}

\begin{proof}
Reconstructed trees under the constant rate birth and death process with $n$ leaves have the same distribution as Yule trees \citep{Aldous2001}. The two daughter trees are denoted by $\cT_1, \cT_2$, they are trees with origin at the $mrca$ and together they have $n$ leaves. The probability that $\cT_1$ has $k$ leaves ($k = 1,2, \ldots n-1$) is $\frac{1}{n-1}$ \citep{Slowinski1990}.  Therefore,
\begin{eqnarray*}
& & \bP[M^{mrca}_{\sigma,t}=m |M^{mrca}_{1,t}=n ]\\
 &=& \frac{1}{n-1} \sum_{k=1}^{n-1} \bP[M^{mrca}_{\sigma,t}=m |M^{\cT_1}_{1,t}=k,M^{\cT_2}_{1,t}=n-k  ]    \\
 &=& \frac{1}{n-1} \sum_{k=1}^{n-1} \sum_{l=1}^{m-1}   \bP[M^{\cT_1}_{\sigma,t}=l, M^{\cT_2}_{\sigma,t}=m-l |M^{\cT_1}_{1,t}=k,M^{\cT_2}_{1,t}=n-k  ]    \\
&=& \frac{1}{n-1} \sum_{k=1}^{n-1} \sum_{l=1}^{m-1}    \bP[M^{\cT_1}_{\sigma,t}=l |M^{\cT_1}_{1,t}=k  ]  \bP[ M^{\cT_2}_{\sigma,t}=m-l |M^{\cT_2}_{1,t}=n-k  ]    \\
 &=& \frac{1}{n-1} \frac{f(\sigma,t,\rho,\delta)^{m-1}}{(1+f(\sigma,t,\rho,\delta)^{n-1}} \sum_{k=1}^{n-1} \sum_{l=1}^{m-1} {k-1 \choose l-1} {n-k \choose m-l}
\end{eqnarray*}
which establishes the theorem.
\end{proof}

\begin{cor}
For $M^{mrca}_{\sigma,t}$, we have the following conditional expectation,
\begin{eqnarray*}
\bE[M^{mrca}_{\sigma,t} |M^{mrca}_{1,t}=n ] &=& \frac{2+n f(\sigma,t)}{ 1+f(\sigma,t) }.
\end{eqnarray*}
\end{cor}

\begin{proof}
The two daugther trees $\cT_1, \cT_2$ of the $mrca$ are trees which have their origin at the $mrca$, together they have $n$ leaves. Since the probability of $\cT_1$ having $k$ leaves ($k = 1,2, \ldots n-1$) is $\frac{1}{n-1}$ \citep{Slowinski1990}, we have
\begin{eqnarray*}
\bE[M^{mrca}_{\sigma,t} |M^{mrca}_{1,t}=n ] &=& \frac{1}{n-1} \sum_{k=1}^{n-1}[ \bE[M_{\sigma,t} |M_{1,t}=k] +  \bE[M_{\sigma,t} |M_{1,t}=n-k]] \\
 &=& \frac{1}{n-1} \sum_{k=1}^{n-1}  \left( 2 +  \frac{(n-2) f(\sigma,t)}{ 1+f(\sigma,t) } \right) \\
&=& \frac{2+nf(\sigma,t)}{ 1+f(\sigma,t) } 
\end{eqnarray*}
which completes the proof.
\end{proof}

The following result had already been established in a completely different way in \citet{Nee1994}. This verifies that our calculations are correct, since we end up with the same result as \citet{Nee1994}.
\begin{cor}
The expected number of species at time $\sigma t$ condition the process survives until t is
$$\bE[M_{\sigma,t} | M_{1,t} > 0] = e^{(\lambda-\mu)\sigma t} \frac{\lambda-\mu e^{-(\lambda-\mu)t}}{\lambda-\mu  e^{-(\lambda-\mu)(1-\sigma t}}.$$ 
\end{cor}
\begin{proof}
We can write the expectation as
\begin{eqnarray*}
\bE[M_{\sigma,t} | M_{1,t} > 0] &=& \sum_{n=1}^{\infty} \bE[M_{\sigma,t} | M_{1,t}=n] \bP[M_{1,t}=n | M_{1,t} > 0] \\
&=& \ldots \\
&=& \sum_{n=1}^{\infty}\frac{1+nf(\sigma,t)}{ 1+f(\sigma,t) } \lambda^{n-1} (\lambda-\mu) \frac{e^{-(\lambda-\mu)t} (1- e^{-(\lambda-\mu)t})^{n-1}}{(\lambda-\mu e^{-(\lambda-\mu)t})^n}\\
&=& \frac{(\lambda-\mu)e^{-(\lambda-\mu)t}}{(1+f(\sigma t))(\lambda-\mu e^{-(\lambda-\mu)t})}
  \sum_{n=0}^{\infty} \left( \frac{\lambda(1-e^{-(\lambda-\mu)t})}{\lambda-\mu e^{-(\lambda-\mu)t}}  \right)^n  \\
  & & + \frac{(\lambda-\mu)e^{-(\lambda-\mu)t}f(\sigma t) }{(1+f(\sigma t))(\lambda-\mu e^{-(\lambda-\mu)t})}
  \sum_{n=1}^{\infty} n \left( \frac{\lambda(1-e^{-(\lambda-\mu)t})}{\lambda-\mu e^{-(\lambda-\mu)t}}  \right)^{n-1} \\
&=& \frac{(\lambda-\mu)e^{-(\lambda-\mu)t}}{(1+f(\sigma t))(\lambda-\mu e^{-(\lambda-\mu)t})}
  \frac{1}{1- \frac{\lambda(1-e^{-(\lambda-\mu)t})}{\lambda-\mu e^{-(\lambda-\mu)t}}  } \\
  & & + \frac{(\lambda-\mu)e^{-(\lambda-\mu)t}f(\sigma t) }{(1+f(\sigma t))(\lambda-\mu e^{-(\lambda-\mu)t})}
  \frac{(\lambda-\mu e^{-(\lambda-\mu)t})^2}{\lambda (\lambda-\mu)^2 e^{-(\lambda-\mu)t}} \frac{d}{dt} \left( \frac{1}{1-\frac{\lambda(1-e^{-(\lambda-\mu)t})}{\lambda-\mu e^{-(\lambda-\mu)t}} }  \right) \\
  &=& \frac{1}{1+f(\sigma t)} + \frac{f(\sigma t) (\lambda-\mu e^{-(\lambda-\mu)t}) }{(1+f(\sigma t)) (\lambda-\mu) e^{-(\lambda-\mu)t}} \\
 &=& e^{(\lambda-\mu)\sigma t} \frac{\lambda-\mu e^{-(\lambda-\mu)t}}{\lambda-\mu  e^{-(\lambda-\mu)(1-\sigma) t}} 
\end{eqnarray*}
which establishes the corollary.
\end{proof}

\section{LTT plots for trees of unknown age}

So far, we assumed that the time since origin is known to be $t$. We then calculate the expected number of species for each point in time between the origin and today.

The fact that the time of origin is known, but nothing about the timing after that seems a bit artifical to me. Aldous/Popovic assumed that any point of time in the past is equally likely to be the point of origin of a tree. Conditioning on $n$ species than gives the distribution $q_{or}(t)$ for the time of origin. I was wondering if we want to write something why it is plausible to use the uniform assumption!?!

If the age of the tree is unknown, we can assume that the age distribution is uniform on $[0,\infty)$. This prior has been assumed before in \citet{AlPo2005, Gernhard2007}.



We will need the following theorem from \citet{GernhardBirthDeath2007}
\begin{thm} \label{ThmPtorn}
Let $t_{or}$ be the time of origin of a tree.
Let $q_{or}(t)$ be the density function of $t_{or}$. Our prior is the uniform distribution of the time of origin on $[0, \infty)$. Conditioning the tree on having $n$ species today, we obtain the following density function for the time of origin of the tree,
$$q_{or}(t|n) =  n \lambda^{n} (\lambda-\mu)^2 \frac{(1-e^{-(\lambda-\mu)t})^{n-1}  e^{-(\lambda-\mu)t}}{(\lambda-\mu e^{-(\lambda-\mu)t})^{n+1}}.$$
\end{thm}

Let $M_{\sigma}$ be the random variable `number of lineages in reconstructed tree when the fraction $\sigma$ of the time until today is over'.
We obtain

\begin{rem}
The probability for $m$ lineages at time $\rho t$ given $n$ species at time $t$ is
$$\bP[M_\sigma = m| M_1 = n] = \int_0^\infty \bP[M_{\sigma,t} = m| M_{1,t} = n] q_{or}(t|n) dt.$$
We did not find an analytic expression for that integral.
\end{rem}
For the expectation, we get
\begin{eqnarray}
& & \bE[M_{\sigma} | M_1=n] \\
 &=& \int_0^{\infty}  \bE[M_{\sigma,t} | M_{1,t}=n] \bP[t|n] dt \notag  \\
&=& \int_0^{\infty} \left( 1+\frac{(n-1)f(\sigma,t)}{1+f(\sigma,t)} \right) \left( n \lambda^{n} (\lambda-\mu)^2 \frac{(1-e^{-(\lambda-\mu)t})^{n-1}  e^{-(\lambda-\mu)t}}{(\lambda-\mu e^{-(\lambda-\mu)t})^{n+1}} \right) dt \notag \\
&=& 1+\int_0^{\infty} \frac{(n-1)f(\sigma,t)}{1+f(\sigma,t)} \left( n \lambda^{n} (\lambda-\mu)^2 \frac{(1-e^{-(\lambda-\mu)t})^{n-1}  e^{-(\lambda-\mu)t}}{(\lambda-\mu e^{-(\lambda-\mu)t})^{n+1}} \right) dt \notag \\
&=& 1+\int_0^{\infty} (n-1) \frac{(\lambda-\mu)(1-e^{-(\lambda-\mu) \sigma t}) e^{-(\lambda-\mu)(1-\sigma)t}  }{(1-e^{-(\lambda-\mu)t})(\lambda-\mu e^{-(\lambda-\mu)(1-\sigma )t})} \ldots \notag \\
& & \left( n \lambda^{n} (\lambda-\mu)^2 \frac{(1-e^{-(\lambda-\mu)t})^{n-1}  e^{-(\lambda-\mu)t}}{(\lambda-\mu e^{-(\lambda-\mu)t})^{n+1}} \right) dt \notag \\
&=& 1+n(n-1)\lambda^n (\lambda-\mu)^3 \int_0^{\infty} \frac{e^{-(\lambda-\mu)(2-\sigma)t}-e^{-(\lambda-\mu)2t} }{\lambda-\mu e^{-(\lambda-\mu)(1-\sigma)t} } \frac{(1-e^{-(\lambda-\mu)t})^{n-2} }{(\lambda-\mu e^{-(\lambda-\mu)t})^{n+1} } dt  \label{EqnForCBP}  \\
&\stackrel{\lambda \neq \mu}{=}& 1+n(n-1)\lambda^n (\lambda-\mu)^2 \int_0^{\infty} \frac{e^{-(2-\sigma)t}-e^{-2t} }{\lambda-\mu e^{-(1-\sigma )t} } \frac{(1-e^{-t})^{n-2} }{(\lambda-\mu e^{-t})^{n+1} } dt \label{EqnForYule}\\
&\stackrel{ \mu \neq 0}{=}& 1+n(n-1) \frac{1}{\lambda \mu} (\lambda-\mu)^2 \int_0^{\infty} \frac{e^{-(2-\sigma)t}-e^{-2t} }{\frac{\lambda}{\mu}- e^{-(1-\sigma )t} } \frac{(1-e^{-t})^{n-2} }{(1-\frac{\mu}{\lambda} e^{-t})^{n+1} } dt \notag \\
&\stackrel{\rho:=\mu / \lambda}{=}& 1+n(n-1)(\rho-1) \int_0^{\infty} \frac{e^{-(2-\sigma)t}-e^{-2t} }{1 - \rho e^{-(1-\sigma )t} } \frac{(1-e^{-t})^{n-2} }{(1- \rho e^{-t})^{n+1} } dt \label{EqnExpt}
\end{eqnarray}
Note that the expectation only depends on $\rho = \lambda / \mu$.
In general, we could not find an analytical solution for the integral. The expected LTT plots are drawn via numerical integration.
However, for the Yule model, $\mu=0$, we can evaluate the integral. From Equation (\ref{EqnForYule}), we get
\begin{eqnarray*}
\bE_{Yule}[M_{\sigma} | M_1=n] &=& 1+n(n-1) \int_0^{\infty} (e^{-(2-\sigma)t}-e^{-2t})  (1-e^{-t})^{n-2}  dt  \\
 &=& 1 + n(n-1) \sum_{k=0}^{n-2} {n-2 \choose k} (-1)^k  \int_0^\infty ( e^{-(k+2-\sigma)t} - e^{-(k+2)t )}dt \\
 &=& 1 + n(n-1)\sum_{k=0}^{n-2} {n-2 \choose k} (-1)^k  \left[-\frac{1}{k+2-\sigma} e^{-(k+2-\sigma)t}    +  \frac{1}{k+2}  e^{-(k+2)t} \right]_0^\infty\\
&=& 1 + n(n-1) \sum_{k=0}^{n-2} {n-2 \choose k} \frac{(-1)^k}{k+2} \frac{\sigma}{k+2-\sigma}.
\end{eqnarray*}
For the critical branching process, i.e. $\lambda = \mu$, we observe with the property $e^{-\epsilon} \sim 1-\epsilon$ for $\epsilon \rightarrow 0$, from Equation (\ref{EqnForCBP}),
\begin{eqnarray*}
& & \bE_{CBP}[M_{\sigma} | M_1=n]\\
 &=& \lim_{\mu \rightarrow \lambda} \left( 1 + n(n-1) \int_0^\infty \frac{\lambda^n (\lambda-\mu)^3 ((\lambda-\mu)\sigma t) ((\lambda-\mu)t)^{n-2} }{(\lambda-\mu(1-(\lambda-\mu)(1-\sigma)t)) (\lambda-\mu(1-(\lambda-\mu)t))^{n+1}} \right)\\
&=& 1 + n(n-1) \int_0^\infty \frac{\lambda^n \sigma t^{n-1}}{(1+\lambda(1-\sigma)t)(1+\lambda t)^{n+1}}dt \\
&=& 1 + n(n-1) \sigma \int_0^\infty \frac{ t^{n-1}}{(1+(1-\sigma)t)(1+ t)^{n+1}}dt.
\end{eqnarray*}
This establishes the following theorem.
\begin{thm}
The expectation of $M_{\sigma}$ given $M_{1}=n$ is
\begin{equation}
 \bE[ M_{\sigma} |M_{1}=n]  =  
 \left\{
\begin{array}{ll}
     1 + n(n-1) \sum_{k=0}^{n-2} {n-2 \choose k} \frac{(-1)^k}{k+2} \frac{\sigma}{k+2-\sigma}  & \hbox{if $\mu = 0$} \\
      1 + n(n-1) \sigma \int_0^\infty \frac{ t^{n-1}}{(1+(1-\sigma)t)(1+ t)^{n+1}}dt        & \hbox{if $\mu = \lambda$} \\
     1+n(n-1)(\rho-1) \int_0^{\infty} \frac{e^{-(2-\sigma)t}-e^{-2t} }{1 - \rho e^{-(1-\sigma )t} } \frac{(1-e^{-t})^{n-2} }{(1- \rho e^{-t})^{n+1} } dt  & \hbox{else} 
\end{array}
\right.
\end{equation}
\end{thm}

Note that $\bE_{Yule}$ and $\bE_{CBP}$ are independent of $\lambda$. The conditioned expectation for $M_\rho$ was calculated for different values of $\rho$, see Figure \ref{FigExpt}. The integration was done with the Matlab~ode45 tool.

\begin{figure}[!h]
\begin{center}
\includegraphics[scale=0.8]{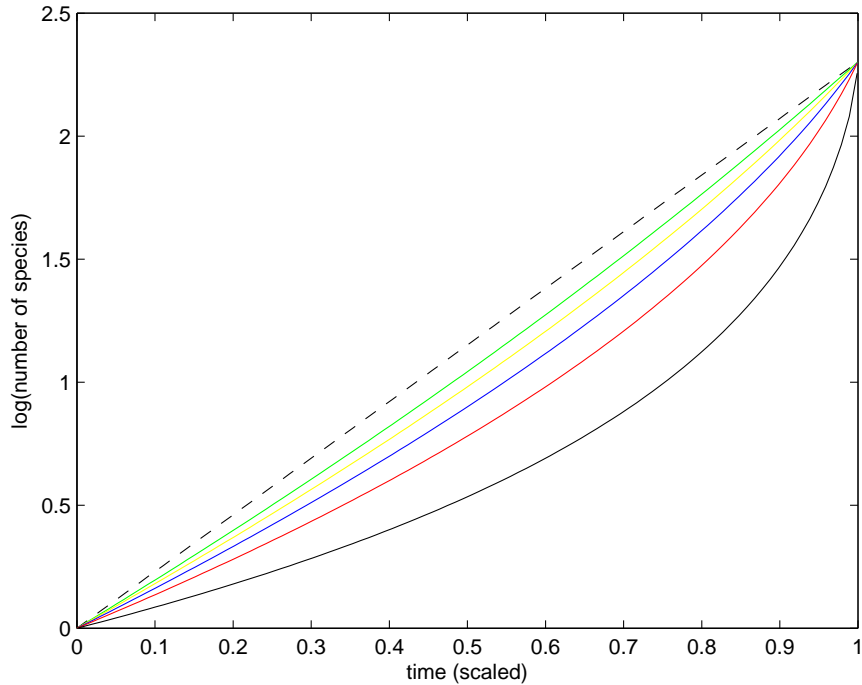}
\caption{Expected number of species given we have $n=10$ species today. According to Equation \ref{EqnExpt}, the expectation only depends on $\rho = \lambda / \mu$, we calculated $\rho=1/4, 1/2, 3/4, 1$ and the Yule model (from bottom to top). The upper black line is the straight line. Note that the Yule model is more convex than the straight line.}
\label{FigExpt}
\end{center}
\end{figure}

\appendix
\section{Proofs}

\bibliographystyle{apalike}
\bibliography{bibliography1}

\end{document}